\documentclass[11pt, reqno]{amsart}

\usepackage{esint}
\usepackage{amstext}
\usepackage{amsthm}
\usepackage{amsmath}
\usepackage{amssymb}
\usepackage{latexsym}
\usepackage{amsfonts}
\usepackage{graphicx}
\usepackage{color}

\usepackage[%pagebackref,
hypertexnames=false, colorlinks, citecolor=black, linkcolor=blue, urlcolor=red]{hyperref}

\newcommand{\ci}[1]{_{ {}_{\scriptstyle #1}}}
\newcommand{\ti}[1]{_{\scriptstyle \text{\rm #1}}}
\newcommand{\R}{\mathbb{R}}

\newcommand{\C}{\mathbb{C}}
\newcommand{\N}{\mathbb{N}}
\newcommand{\F}{\mathbb{F}}

\newcommand{\I}{\mathbf{I}}

\newcommand{\dd}{\mathrm{d}}

\newcommand{\cF}{\mathcal{F}}
\newcommand{\cX}{\mathcal{X}}
\newcommand{\cH}{\mathcal{H}}

\newcommand{\cD}{\mathcal{D}}
\newcommand{\cB}{\mathcal{B}}

\newcommand{\1}{\mathbf{1}}

\newcommand{\fdot}{\,\cdot\,}
\newcommand{\wt}{\widetilde}

\newcommand{\bW}{\mathbf{W}}
\newcommand{\bw}{\mathbf{w}}

\newcommand{\La}{\langle }
\newcommand{\Ra}{\rangle }
\newcommand{\rk}{\operatorname{rk}}

\newcommand{\im}{\operatorname{Im}}
\newcommand{\re}{\operatorname{Re}}
\newcommand{\tr}{\operatorname{tr}}

\newcommand{\be}{\begin{equation}}
\newcommand{\ee}{\end{equation}}

\newcommand{\sd}{{\scriptstyle\Delta}}

\newcommand{\ch}{\operatorname{ch}}

\renewcommand{\labelenumi}{(\roman{enumi})}

\newcounter{vremennyj}

\newcommand\cond[1]{\setcounter{vremennyj}{\theenumi}\setcounter{enumi}{#1}\labelenumi\setcounter{enumi}{\thevremennyj}}

\numberwithin{equation}{section}

\newtheorem{thm}{Theorem}[section]
\newtheorem{lm}[thm]{Lemma}

\newtheorem*{prop*}{Proposition}

\theoremstyle{remark}

\newtheorem*{rem*}{Remark}

\begin{document}

\title{The Carleson embedding theorem with matrix weights}

\author{Amalia Culiuc}

\address{Amalia Culiuc: Department of Mathematics, Brown University, 
%151 Thayer Str./Box 1917,      
 Providence, RI  02912, USA }
\email{amalia@math.brown.edu}

\author{Sergei Treil}
 \thanks{S.~Treil is supported  in part by the National Science Foundation under the grant DMS-1301579.}
\address{Sergei Treil: Department of Mathematics \\ Brown University \\ Providence, RI 02912 \\ USA}
\email{treil@math.brown.edu}



\begin{abstract}
In this paper we prove the weighted martingale Carleson embedding theorem with matrix weights both in the domain and in the target space. 
\end{abstract}

\maketitle

\section{Introduction and main results}

The main result of this paper is the matrix weighted martingale Carleson embedding theorem, where matrix weights appear in both the domain and the target space. The need for such result is motivated by the  attempt to generalize the two weight estimates for well localized  operators from \cite{NTV-2weight-2008},  to the case of matrix-valued measures. The main part of the estimate in \cite{NTV-2weight-2008} is the two weight inequality for paraproducts, and this estimate for the matrix-valued measures can be reduced exactly to the embedding theorem treated in this paper.

Earlier versions of the matrix weighted Carleson Embedding Theorem theorem under fairly strong additional assumptions (such as the weight belonging to the $A_2$ class) go back to \cite{Wav_PastFuture} and, more recently, \cite{IKP},  \cite{BW}. Two weight estimates with matrix weighs for well-localized operators, also under additional assumptions, were treated in \cite{kerr} and \cite{BW} (see also \cite{BPW}), but the result was still not known in full generality.
%under the additional assumption of the matrix weights belonging to the $A_2$ class and with embedding constants depending on the $A_2$ and Reverse H\"older characteristics. 

The weighted embedding theorem presented in this paper does not assume any  properties for the matrix weight except local boundedness, and produces an embedding constant that depends polynomially on the dimension of the space.
As in the scalar case, our embedding theorem states the Carleson measure condition, which is just a simple testing condition, implies the embedding. For matrix weights the Carleson measure condition (condition \cond2 in Theorem \ref{t:MCET-01} or condition \cond3 in Theorem \ref{main}) is an inequality between positive semidefinite matrices. 

In the case of scalar weights in the domain, the right hand side of the inequality is a multiple of the identity matrix $\I$: in this situation, sacrificing constants, one can replace matrices by their norms, and the matrix embedding theorem trivially follows from the scalar one. Of course, the constants obtained by such trivial reduction are far from optimal:  constants of optimal order were obtained using more complicated reasoning in \cite{NaPiTrVo}. 

In our case, both sides of the Carleson measure condition are general positive semidefinite matrices, so the simple strategy of replacing matrices by norms or traces does not work. A more complicated idea, in the spirit of the argument in \cite{NaPiTrVo}, is used to get the result.

\subsection{Setup}
\subsubsection{Atomic filtered spaces}

Let $(\mathcal{X}, \cF, \sigma)$ be a sigma-finite measure space with an atomic filtration $\mathcal{F}_n$, that is, a sequence of increasing sigma-algebras $\cF_n\subset \cF$ such that for each $\mathcal{F}_n$ there exists a countable collection $\mathcal{D}_n$ of disjoint sets of finite measure with the property that every set of $\mathcal{F}_n$ is a union of sets in $\mathcal{D}_n$. 

We will call the sets $I\in\cD_n$ \emph{atoms}, and denote by $\cD$ the collection of all atoms,  $\displaystyle\mathcal{D}=\cup_{n\in \mathbb{Z}}\mathcal{D}_n$. We allow a set $I$ to belong to several generations $\cD_n$, so formally an atom $I\in\cD_n$ is a pair $(I,n)$. To avoid overloading the notation, we skip the ``time'' $n$ and write $I$ instead of $(I,n)$; if we need to ``extract'' the time $n$, we will use the symbol $\rk I$. Namely, if $I$ denotes the atom $(I,n)$ then $n=\rk I$. 

The inclusion $I\subset J$ for atoms should be understood as inclusion for the sets together with the inequality $\rk I \ge \rk J$. However, the union (intersection) of atoms is just the union (intersection) of the corresponding sets and ``times'' $n$ are not taken into account. 

A standard example of such a filtration is the dyadic lattice $\mathcal{D}$ on $\mathbb{R}^N$, which explains the choice of notation. However, in what follows, $\mathcal{D}$ will always denote a general collection of atoms and $I\in \mathcal{D}$ will stand for an atom  in $\mathcal{D}$, and not necessarily for a dyadic interval. 

\subsubsection{Matrix-valued measures}
%Define
%\[
%\cF_0:=\{E\cap F: E\in \cF, \, F=\bigcup_1^n I_k, \ I_k\in\cD\}. 
%\] 
Let $\cF_0$ be the collection of sets  $E\cap F$ where  $E\in\cF$ and $F$ is a finite union of atoms.  
A $d\times d$ matrix-valued measure $\bW$ on $\cX$, is a countably additive function on $\cF_0$ with values in the set of $d\times d$ positive semidefinite matrices. Equivalently, $\bW=(\bw_{j,k})_{j,k=1}^d$ is a $d\times d$ matrix whose entries $\bw_{j,k}$ are (possibly signed or even complex-valued) measures, finite on atoms, and such that for any $E\in\cF_0$ the matrix    $(\bw_{j,k}(E))_{j,k=1}^d$ is positive semidefinite. Note that the measure $\bW$ is always finite on atoms. 

The weighted space $L^2(\bW)$ is defined as the set of all measurable $\F^d$-valued functions ($\F=\R$ or $\C$) such that 
\[
\|f\|\ci{L^2(\bW)}^2 := \int_\cX \Bigl\La \bW(\dd x) f(x), f(x)\Bigr\Ra_{\F^d} <\infty ;
\]
as usual we take the quotient space over the set of functions of norm $0$.

\subsection{Main results}
\begin{thm}
\label{t:MCET-01}
Let $\bW$ be a $d\times d$ matrix-valued measure and let $\wt A\ci I$ be positive semidefinite $d\times d$ matrices. The following statements are equivalent:
\begin{enumerate}
\item  $\displaystyle \sum_{I\in \cD} \left\| \wt A^{1/2}\ci I \int_I \bW(\dd x)f(x) \right\|^2 \le A \|f\|\ci{L^2(\bW)}^2$ for all $f\in L^2(\bW)$;
\item $\displaystyle \sum_{\substack{I\in \cD \\ I\subset I_0}}  \bW(I) \wt A\ci I \bW(I) \le B \bW(I_0) $ for all $I_0\in\cD$. 
\end{enumerate}
Moreover, for the best constants $A$ and $B$ we have $B\le A \le C B$, where $C=C(d)=4e^2 d^2$. 
\end{thm}

Note that the underlying measure $\sigma$ is absent from the statement of the theorem: we do not need $\sigma$ in the setup, we only need the filtration $\cF_n$. Alternatively, we can pick $\sigma$ to make the setup more convenient. For example, if we define $\sigma:= \tr \bW :=\sum_{k=1}^d \bw_{k,k}$, 
%where $|w|$ stands for the total variation of the measure $w$, 
then the measures $\bw_{j,k}$ are absolutely continuous with respect to $\sigma$. Thus, we can always assume that our matrix-valued measure $\bW$ is an absolutely continuous one $W\dd\sigma$, where $W$ is a matrix weight, i.e.~a locally integrable (meaning integrable on all atoms $I$) matrix-valued function with  values in the set of positive semidefinite matrices. 

For a measurable function $f$ we denote by $\La f\Ra\ci I$ its average, 
\[
\La f\Ra\ci I := \sigma(I)^{-1} \int_I f \dd\sigma;
\]
if $\sigma(I)=0$ we put $\La f\Ra\ci I=0$. The same definition is used for the vector and matrix-valued functions.  

In what follows we will often use $|E|$ for $\sigma(E)$ and $\dd x$ for $\dd\sigma$. 

The theorem below is the restatement of Theorem \ref{t:MCET-01} in this setup, if we put $A\ci I = |I|^{-1} \wt A\ci I$. More precisely, Theorem \ref{t:MCET-01} is just the equivalence \cond2$\iff$\cond3 in Theorem \ref{main}. The equivalence \cond1$\iff$\cond2 will be explained below.

\begin{thm}
\label{main}
Let $W$ be a $d\times d$ matrix-valued weight and let $A\ci I$, $I\in\cD$ be a sequence of positive semidefinite $d\times d$ matrices. Then the following are equivalent:
\begin{enumerate}
\item $\displaystyle \sum_{I\in \mathcal{D}}\left\|A\ci I^{1/2}\La W^{1/2}f\Ra\ci I\right\|^2|I|\le A \|f\|^2\ci{L^2}$.
\item $\displaystyle \sum_{I\in \mathcal{D}}\left\|A\ci I^{1/2}\La W f\Ra\ci I\right\|^2|I|\le A \|f\|^2_{L^2(W)}$.
\item $ \displaystyle \frac{1}{|I_0|} \sum_{\substack{I\in\cD\\ I\subset I_0}}\La W\Ra \ci IA\ci I\La W\Ra \ci I |I|\le B\La W\Ra\ci{I_0} $ for all $I_0\in \mathcal{D}.$
\end{enumerate}
Moreover, the best constants $A$ and $B$ satisfy $B\le A \le C B$, where $C=C(d)=4e^2 d^2$.  
%is a constant depending only on the dimension $d$. 
\end{thm}

\subsection*{Thanks} The authors would like to thank F.~Nazarov for suggesting the crucial idea of using a ``Bellman function with a parameter'' argument, similar to one used earlier in \cite{NaPiTrVo}. We also would like to thank F.~Nazarov and M.~Sodin for suggesting Lemma \ref{l:poly1} instead of a more complicated reasoning with Jacobi polynomials, that was used in the earlier version of the paper.  

\section{Proof of the main result}

\subsection{Trivial reductions}

The equivalence of \cond1 and \cond2 in Theorem \ref{main} is trivial. In \cond1, perform the change of variables $f:=W^{1/2}f$ to obtain \cond2 and similarly, in \cond2 set $f:=W^{-1/2}f$ to obtain \cond 1. 
Note that here we do not need to assume that the weight $W$ is invertible a.e.: 
we just interpret $W^{-1/2}$ as the Moore--Penrose inverse of $W^{1/2}$.

The implication  \cond1$\implies$\cond3  and the estimate $A\ge B$ are obvious by setting $f=W^{1/2}\1\ci{I_0} e$, $e\in \F^d$ in \cond1. Equivalently, to show that \cond2$\implies$\cond3 one just needs to apply \cond2 to the test functions $f=\1\ci{I_0} e$.

So it only remains to prove that \cond3$\implies$\cond1, or equivalently, that \cond3$\implies$\cond2.

%To prove that \cond3$\implies$\cond1, it is enough to show that for any atom $I_0$ 
%%of fixed size,
%%%
%\begin{align}
%\label{MainEst}
% \displaystyle \sum_{\substack{I\in \mathcal{D}\\I\subset I_0}}\left\|A\ci I^{1/2}\La W^{1/2}f\Ra\ci I\right\|\leq C B\|f\|^2_{L^2}.
%\end{align}
%%%
%Indeed, applying this estimate to all $I_0\in \cD_{-N}$, we get that \cond1 holds if we take summation only over $I\in \bigcup_{n\ge -N}\cD_n$; letting $N\to \infty$ finishes the proof. 

\subsubsection{Invertibility of \texorpdfstring{$W$}{W}}
Let us notice that without loss of generality we can assume that the weight $W$ is invertible a.e., and even that the weight $W^{-1}$ is uniformly bounded.  

To show that, define for $\alpha>0$ the weight $W_\alpha$ by $W_\alpha (s) := W(s) + \alpha \I$, and let
\[
A\ci I^\alpha := \La W_\alpha\Ra\ci I^{-1} \La W\Ra\ci I A\ci I \La W\Ra\ci I \La W_\alpha\Ra\ci I^{-1}\,.
\]
If \cond3 is satisfied, then trivially 
\[
\frac{1}{|I_0|} \sum_{\substack{I\in\cD\\ I\subset I_0}}\La W_\alpha\Ra \ci IA\ci I^\alpha \La W_\alpha\Ra \ci I |I|\le B\La W\Ra\ci{I_0} \le B \La W_\alpha\Ra\ci{I_0}\,.
\]
If Theorem \ref{main} holds for invertible weights $W$, we get that for all $f\in L^2(W)\cap L^2$
\[
\sum_{I\in \mathcal{D}}\left\|(A\ci I^\alpha)^{1/2}\La W_\alpha f\Ra\ci I\right\|^2|I|\le A \|f\|^2_{L^2(W_\alpha)} . 
\]
Noticing that $\|f\|\ci{L^2(W_\alpha)} \to \|f\|\ci{L^2(W)}$, $  \La W_\alpha f\Ra\ci I \to \La W f\Ra\ci I$, $A\ci I^\alpha \to A\ci I$ as $\alpha\to0^+$ we immediately get \cond2 for all $f\in L^2(W)\cap L^2$ and for all finite sums (i.e.~for the case of finitely many non-zero $A\ci I$). The  case of infinite sums follows immediately, 
since the infinite sum is a limit of finite sums.

Since the estimate \cond2 holds on a dense set, extending the embedding operator by continuity we trivially get that  it holds for all  $f\in L^2(W)$.

%Indeed, since an infinite sum of non-negative numbers is the supremum of all finite subsums, it is sufficient to prove  Theorem \ref{main} for the case when only finitely many matrices $A\ci I$ are non-zero. 

%Let $\cA\subset \cD$ be the  collection of all atoms $I$ such that $A\ci I\ne 0$ and $|I|\ne0$, so only $I\in \cD$ should be considered in the sums. Recall that we assumed that $\cA$ is finite.
%
%Condition \cond3 holds if and only if for all $\e>0$
%%%
%\begin{align}
%\label{3a}
%\frac{1}{|I_0|} \sum_{\substack{I\in\cA\\ I\subset I_0}}\La W\Ra \ci IA\ci I\La W\Ra \ci I |I|\le B\La W\Ra\ci{I_0} +\e \I .
%\end{align}
%%%
%
%
%  We want to show that if \cond3 is satisfied, then for any $\e>0$ estimate \eqref{3a} holds for $W_\alpha$ for all sufficiently small $\alpha$. 
%
% We can trivially estimate for $o<\alpha\le 1$
%%%
%\begin{align*}
%\sum_{\substack{I\in\cA\\ I\subset I_0}}\La W_\alpha\Ra \ci IA\ci I\La W_\alpha\Ra \ci I |I|
%&\le 
%\sum_{\substack{I\in\cA\\ I\subset I_0}}\La W\Ra \ci IA\ci I\La W\Ra \ci I |I| + \alpha M \I
%\intertext{where }
%M &= \sum_{I\in\cA} \left(2\|A\ci I\| \| \La W\Ra \ci I\| + \|A\ci I\|\right)|I| 
%\end{align*}
%%%
%
%Note that it is sufficient 

\subsection{The Bellman functions}
By homogeneity we can assume without loss of generality that $B = 1$. As we discussed above, we only need to prove the implication \cond3$\implies$\cond1.

Following a suggestion by F.~Nazarov we will do so by a ``Bellman function with a parameter'' argument similar to one presented in \cite{NaPiTrVo}.
Denote
\begin{align}
\label{F_I}
F\ci I&=\|f\|^2\ci{L^2(I)} := \La |f|^2 \Ra\ci I\\
\label{M_I}
M\ci I&= \frac{1}{|I|} \displaystyle \sum_{J\subset I}\La W\Ra\ci J A\ci J\La W\Ra\ci J |J|\\
\label{x_I}
x\ci I&=\La W^{1/2}f\Ra\ci I.
\end{align}
For any real number $s$, $0\leq s<\infty$, define the Bellman function
\begin{align}
\label{B_s}
\cB_s(I)=\cB_s(F\ci I,x\ci I, M\ci I)= \Bigl\La \left(\La W\Ra \ci I+sM\ci I\right)^{-1}x\ci I, x \ci I\Bigr\Ra_{\F^d} .
\end{align}
Notice that $F\ci I$ is not involved in the definition of $\cB_s(I)$, but it will be used in the estimates. 

The functions $\cB\ci I$ satisfy the following properties:
%
%\subsubsection{Properties of the Bellman function}
%The Bellman function $B_s$ satisfies the following properties.
\begin{enumerate}
\item The range property: $0\le \cB_s(I)\le F\ci I$;
\item The key inequality:
\begin{align}
\label{key}
\cB_s(I) + s R\ci I(s) &\le \sum_{I'\in\ch(I)}\frac{|I'|}{|I|} \cB_s(I'), \qquad s\ge 0,
\intertext{where}
\notag
R\ci I(s) &=  \| A\ci I^{1/2} \La W\Ra\ci I (\La W\Ra\ci I + s M\ci I)^{-1} x\ci I \|^2 .
\end{align}

\end{enumerate}
The inequality  $\cB_s(I)\ge 0$ is trivial, and the inequality $\cB_s(I)\leq F\ci I$ follows immediately from 
%through an application of the Cauchy-Schwarz inequality. The details of this computation are presented in the proof of 
Lemma \ref{range} below. 
The key inequality \eqref{key} is a consequence of Lemma \ref{l:convex}, which we also prove below.

\subsection{From Bellman functions to the estimate}
Let us rewrite \eqref{key} as 
\[
|I| \cB_s(I) +|I| s R\ci I(s) \le \sum_{I'\in\ch(I)}|I'| \cB_s(I').  
\]
Then, applying this estimate to each $\cB_s(I')$, and then to each  descendant of each $I'$, we get, going $m$ generations down,
\begin{align*}
|I|\cB_s(I) + \sum_{\substack{I'\in\cD:I'\subset I \\ \rk I' < \rk I + m }} s  R\ci{I'}(s) |I'| 
&\le  \sum_{\substack{I'\in\cD:I'\subset I \\ \rk I' = \rk I + m }} |I'| \cB_s(I') 
 \le \| f\1\ci I\|\ci{L^2}^2; 
\end{align*}
here in the last inequality we used the fact that  $\cB_s(I) \le F\ci I = \La |f(\fdot)|^2\Ra\ci I =|I|^{-1}\|f\1\ci I\|\ci{L^2}^2$. 

Letting $m\to\infty$ and ignoring the non-negative term $|I|\cB_s(I)$ in the left hand side, we get that
\[
s \sum_{\substack{I'\in\cD:\, I'\subset I}}   R\ci{I'}(s) |I'| \le \| f\1\ci I\|\ci{L^2}^2. 
\]
Summing the above  inequality over all $I\in\cD_n$  we obtain
\[
s \sum_{\substack{I'\in\cD:\, \rk I'\ge n }}   R\ci{I'}(s) |I'| \le \| f\|\ci{L^2}^2. 
\]
Then, letting $n\to-\infty$ and replacing $I'$ by $I$, we arrive to the estimate
\begin{equation}\label{rbound}
s\sum_{I\in\cD} R\ci{I}(s) |I| \le \|f\|\ci{L^2}^2. 
\end{equation}

Note that $R\ci I(0) = \|A\ci I^{1/2} x\ci I\|^2 = \|A\ci I^{1/2} \La W^{1/2} f\Ra\ci I \|^2$, so to prove \cond1 we need to estimate $\sum_I R\ci I(0)|I|$, and we only have the estimate of $s\sum_{I} R\ci I(s)|I|$!

In the scalar case we trivially (since $M\ci I\le \La W\Ra\ci I$) have $R\ci I(0) \le 4 R\ci I(1)$, which gives us \cond1 with constant $4B$.

But due to non-commutativity, such estimate fails in the matrix case, so an extra trick is needed.

First, observe that it follows from the cofactor inversion formula that the entries of the matrix $(\La W\Ra\ci I + s M\ci I)^{-1}$ are of the form ${p_{j,k}(s)}/{Q(s)}$, where 
\[
Q(s)=Q\ci I (s)=\det(\La W\Ra\ci I + s M\ci I) 
\] 
is a polynomial of degree at most $d$, and
$p_{j,k}(s)$ are polynomials of degree at most $d-1$.

Therefore $R\ci I$ is a rational function of $s\in\R$, $R\ci I(s)={{P}\ci I(s)}/{|Q\ci I(s)|^2}$, where ${P}\ci I(s)$ is a polynomial of degree at most $2(d-1)$,  
${P}\ci I(s)\geq 0$ for $s\ge 0$.

Denote $p\ci I := |Q\ci I(0)|^{-2} P\ci I$. We will show later that estimate \eqref{rbound} implies the following statement.

\begin{lm}
\label{l:pbound} For $s>0$ and $p\ci I$ introduced above
\begin{align}
\label{pbound}
\sum_{I\in \cD} p\ci I(s) |I| \le s^{-1} (1+s)^{2d} \|f\|\ci{L^2}^2
\end{align}
\end{lm}
 
We will need the following simple lemma. 
%This lemma was suggested to the authors by F.~Nazarov and M.~Sodin, instead of a more complicated reasoning involving Jacobi polynomials, that was used in an earlier version of the paper.  
%%
\begin{lm}
\label{l:poly1}
Let $p$ be a polynomial satisfying for some $n\in\N$
\begin{align}
\label{poly-est}
|p(s)| \le s^{-1} (1+s)^n \qquad \forall s >0. 
\end{align}
Then 
\begin{align}
\label{poly-est-01}
|p(0)| \le e^2 n^2
\end{align}
\end{lm}
Applying  Lemma \ref{l:poly1} to partial sums in \eqref{pbound} with $n=2d$ we get that 
\begin{align*}
\sum_{I\in\cD} p\ci I(0) |I| \le 4e^2 d^2 \|f\|\ci{L^2}. 
\end{align*}
But $p\ci I (0) = R\ci I (0)$, so 
\begin{align*}
\sum_{I\in\cD} R\ci I(0) |I| \le 4e^2 d^2 \|f\|\ci{L^2}^2, 
\end{align*}
which is exactly statement \cond1 of Theorem \ref{main}. So, the theorem is proved (modulo lemmas \ref{l:pbound} and \ref{l:poly1}). \hfill\qed

\begin{proof}[Proof of Lemma \ref{l:pbound}]
By hypothesis,  $M\ci I\le \La W\Ra\ci I$, so the operator $\La W\Ra\ci I + s M\ci I$ is invertible for all $s$ such that $\re(s)> -1$. Thus the zeroes of $Q\ci I(s)$ are all in the half plane $\re(s)\leq -1$. Let $\lambda_1, \lambda_2, ..., \lambda_d$ be the roots of the polynomial $Q\ci I(s)$ counting multiplicity. We have
\[
\left|\frac{Q\ci I(s)}{Q\ci I (0)}\right|=\prod_{k=1}^d\left|\frac{s-\lambda_k}{\lambda_k}\right|.
\]
For a fixed $s>0$  the term $|s-\lambda_k|/|\lambda_k|$, $\re \lambda_k \le -1$  attains its maximum at $\lambda_k=-1$. Therefore $|Q(s) |/|Q(0)|\le (1+s)^d$ for all $s>0$, so we get from \eqref{rbound} that for all $s>0$
\begin{align*}
s\sum_{I\in\cD} p\ci I (s) |I| = s\sum_{I\in\cD} \frac{|Q\ci I(s)|^2}{|Q\ci I(0)|^2} R\ci I(s) |I| 
\le (1+s)^{2d} \cdot s\sum_{I\in\cD} R\ci I(s) |I| \le (1+s)^{2d}\|f\|\ci{L^2}^2, 
\end{align*}
which is equivalent to \eqref{pbound}. 
\end{proof}

\begin{proof}[Proof of Lemma \ref{l:poly1}]
Define the polynomial $q$ by $q(z) = p(z^2)$. Then for all $s\in \R$
\begin{align}
\label{q-est}
s^2 |q(s) |\le (1+s^2)^n. 
\end{align}
We want to show that for any polynomial $q$, the estimate \eqref{q-est} implies that $|q(0)|\le e^2n^2$. 
Since replacing a root of $q$ by its conjugate does not change $|q(s)|$ on the real line $\R$,  we can assume without loss of generality that all roots of $q$ are in the lower half-plane $\im z \le 0$. 

But then trivlally
\begin{align}
\label{q(it)}
|q(0)|\le |q(it) | \qquad\forall t>0. 
\end{align}

Note, that the condition \eqref{q-est} implies that degree of the polynomial $q$ is at most $2n-2$. 
Therefore, the functions $z\mapsto (1-iz)^{-2n} z^2q(z)$ and $F$, $F(z) =(1-iz)^{-2n}(1+z^2)^n$ belong to the  class $H^\infty(\C_+)$ of bounded analytic functions in the upper half-plane $\C_+:=\{z\in\C:\im z >0\}$ and satisfy 
\begin{align*}
|(1-is)^{-2n} s^2 q(s)| \le |F(s)| \qquad \forall s\in\R. 
\end{align*}
The inner-outer factorization of $H^\infty$ functions implies then that for all $z\in \C_+$
\begin{align}
\label{q<Fo}
|(1-is)^{-2n} z^2 q(z) | \le |F\ti o (z) | , 
\end{align}
where $F\ti o$ is the outer part of $F$. The outer part $F\ti o$ is easily computed: the function
\begin{align*}
F(z) = (1-iz)^{-2n}(1+z^2)^n = (1-iz)^{-2n} (1-iz)^n (1+iz)^n =  \frac{(1+iz)^n}{(1-iz)^n}= F\ti o (z) 
\end{align*}
 is a Blaschke product, so $F\ti o(z) \equiv 1 $. Substituting $t=1/n$ into \eqref{q<Fo} we get that 
\begin{align*}
n^{-2} |q(i/n)| \le (1+1/n)^{2n} \le e^2, 
\end{align*}
which by \eqref{q(it)} implies 
\begin{align*}
|p(0)| =|q(0)| \le |q(i/n)| \le e^2 n^2. 
\end{align*}
\end{proof}

\section{Verifying properties of \texorpdfstring{$\cB_s$}{B<sub>s}}

It remains to show that $\cB_s$ satisfies the Bellman function properties. The range property \cond1 is proved in the following proposition:

\begin{lm}\label{range}
For $\cB_s$ defined above in \eqref{B_s},
 $\cB_s(I)\le F\ci I$.
\end{lm}

\begin{proof}
%Indeed, notice that for any fixed $I$ the minimal element of the family 
%$B_s(I)$ is $B_0(I)=F\ci I-\La \La W\Ra _I^{-1}x\ci I, x\ci I\Ra$. 
Let $e\in\mathbb{F}^d$. Since $W$ is self-adjoint, an application of the Cauchy-Schwarz inequality gives
\[
\left|\fint_I\La W^{1/2} f, e\Ra \right|\leq \left(\fint_I\La f,f\Ra \right)^{1/2}\left(\fint_I\La W^{1/2}e,W^{1/2}e\Ra \right)^{1/2}.
\]
Therefore, recalling  the  notation \eqref{F_I}, \eqref{x_I}, we get that for any vector $e$,
\begin{equation}\label{eineq}
\frac{\left|\La x\ci I,e\Ra \right|^2}{\La \La W\Ra\ci I e,e\Ra }\leq F\ci I.
\end{equation}
Using Lemma \ref{l:A^{-1}} below we can write 
\begin{align*}
\La (\La W \Ra\ci I + s M\ci I)^{-1} x, x\Ra  & = \sup_{e\ne 0}
\frac{|\La x\ci I, e\Ra|^2}{\La (\La W\Ra\ci I + s M\ci I) e, e\Ra } \\
& \le \sup_{e\ne 0} \frac{\left|\La x\ci I,e\Ra \right|^2}{\La \La W\Ra\ci I e,e\Ra } \le F\ci I \, ,
\end{align*}
%%
%%
%%
%\Let $a:=\La W\Ra _I^{1/2}e$. The inequality \ref{eineq} can be rewritten as
%begin{equation*}
%\frac{\left|\La x\ci I,\La W\Ra _I^{-1/2}a\Ra \right|^2}{\|a\|^2}\leq F\ci I
%\]
%%%
%for all nonzero vectors $a$ in $\mathbb{R}^n$. 
%%and in particular,
%%%
%%\[
%%\sup_{a}\frac{\left|\La x\ci I,\La W\Ra \ci I^{-1/2}a\Ra \right|^2}{\|a\|^2}\leq %F\ci I.
%%\]
%%%
%Substituting $a=\La W\Ra\ci I^{-1/2}x\ci I$ we get 
%\[
%\La \La W\Ra\ci I^{-1}x\ci I,x\ci I\Ra \leq F\ci I, 
%\]
%%%
%so
%%%
%\[
%B_s(I)\geq B_0(I)\geq 0.
%\]
%%%
which means exactly that $\cB_s(I)\le F\ci I$. 
\end{proof}

\begin{lm}
\label{l:A^{-1}}
Let $A\ge 0$  be an invertible operator in a Hilbert space $\cH$. Then for any vector $x\in\cH$ 
\[
\La A^{-1} x, x\Ra =\sup_{e\in\cH:\,e\ne 0} \frac{|\La x, e\Ra|^2}{\La A e, e\Ra}
\]
\end{lm}

\begin{proof}
By definition, 
\begin{align*}
\La A^{-1}x,x\Ra &=\|A^{-1/2}x\|^2\\
&=\sup_{a\in\cH:\,\|a\|\neq 0}\frac{|\La A^{-1/2}x,a \Ra|^2}{\|a\|^2}\\
&=\sup_{a\in\cH:\,\|a\|\neq 0}\frac{|\La x,A^{-1/2}a \Ra|^2}{\|a\|^2}.
\end{align*}
Making the change of variables  $a=A^{1/2}e$ we conclude
\begin{align*}
\La A^{-1}x,x\Ra &=\sup_{e\in\cH:\,\|e\|\neq 0}\frac{|\La x,e \Ra|^2}{\La Ae,e\Ra}.
\end{align*}
\end{proof}

The main estimate \eqref{key} is the consequence of the following lemma:
\begin{lm}
\label{l:convex}
Let $\cH$ be a Hilbert space. For $x\in\cH$ and for $U$ being a bounded invertible positive operator in $\cH$ define 
\[
\phi(U,x):= \La U^{-1} x,x\Ra\ci{\cH} . 
\]
Then the function $\phi $ is convex, and, moreover, if 
\begin{align}
\label{mart-prop-01}
x_0=\sum_{k\ge 1} \theta_k x_k, \qquad \sd U:= U_0 -\sum_{k\ge 1} \theta_k U_k
\end{align}
where $0\le\theta_k\le 1$, $\sum_{k\ge 1}\theta_k =1$, then 
\begin{align}
\label{key-02}
\sum_{k\ge 1} \theta_k \phi(U_k, x_k) - \phi(U_0,x_0) \ge \La U_0^{-1}\sd U U_0^{-1} x_0,x_0\Ra\ci\cH \,.
\end{align}
\end{lm}

To see that this lemma implies \eqref{key}, fix $s>0$.  Denoting 
\begin{align*}
U\ci I^s =\La W\Ra\ci I + s M\ci I, \qquad x\ci I =\La W^{1/2} f\Ra\ci I, 
\end{align*}
we see that 
\[
B_s(I) =  \phi(U\ci I^s, x\ci I). 
\]
Let $I_k$, $k\ge 1$ be the children of $I$, and let $\theta_k=|I_k|/|I|$.  Notice that 
$\La W\Ra\ci I = \sum_{k\ge 1} \theta_k \La W\Ra\ci{I_k}$, $M\ci I =\sum_{k\ge 1} \theta_k M\ci{I_k} + s \La W\Ra\ci I A\ci I \La W\Ra\ci I$,  so
\[
U\ci I^s -\sum_{k\ge 1} \theta_k U\ci{I_k}^s =:\sd U^s = s \La W\Ra\ci I A\ci I \La W\Ra\ci I.
\]

Therefore, applying Lemma \ref{l:convex} with $U_0=U\ci I^s$, $x_0=x\ci I$, $U_k=U^s\ci{I_k}$, $x_k=x\ci{I_k}$, $\sd U = \sd U^s$ we  get  \eqref{key-02}, that translates   exactly to the estimate \eqref{key}. 
%
% 
%%%
%\begin{align}
%\label{key-01}
%\sum_k \theta_k \phi(U^s\ci{I_k} , x\ci{I_k}) \ - \ \phi(U\ci I,x\ci I) 
%\ge 
%s\|A\ci I^{1/2} \La W\Ra\ci I ( \La W\Ra\ci I+sM\ci I)^{-1}\|^2 .
%\end{align}
%%%
%Recalling that $F\ci I =\sum_k \theta_k F\ci{I_k}$ we get \eqref{key} by multiplying both sides of \eqref{key-01} (and changing the sign of the inequality). 
%

\begin{proof}[Proof of Lemma \ref{l:convex}]
The function $\phi$ and the right hand side of \eqref{key-02} are invariant under the change of variables
\begin{align}
\label{ChVar-01}
x\mapsto U_0^{-1/2}x, \qquad U\mapsto U_0^{-1/2}UU_0^{-1/2}, 
%\qquad \sd U\mapsto U_0^{1/2}x, \qquad  $U=U_0^{-1/2}\sd UU_0^{-1/2} 
\end{align}
so it is sufficient to prove \eqref{key-02} only for $U_0=\I$.

%Assume without loss of generality that $U_0=\I$ (set $x:=U_0^{1/2}x$, 
%$U=U_0^{-1/2}UU_0^{-1/2}$ and $\sd U:=U_0^{1/2}x$, $U=U_0^{-1/2}\sd UU_0^{-1/2}$). 
In this case define  function ${\varPhi}(\tau)$, $0\le\tau\le1$ as
\[
\varPhi(\tau)=\sum \theta_k\Bigl\La \left(\I+\tau\sd U_k\right)^{-1}(x_0+\tau \sd x_k), (x_0+\tau \sd x_k)\Bigr\Ra_\cH -\La x_0,x_0\Ra\ci\cH ,
\]
where $\sd x_k=x_k-x_0$ and $\sd U_k=U_k-U_0=U_k-\I$. Using the power series expansion of $(\I +\tau\sd U)^{-1}$ we get
\begin{align*}
\varPhi(\tau)=&\tau \left(2\sum\theta_k\La \sd x_k,x_0\Ra\ci\cH -\sum\theta_k\La \sd U_k x_0,x_0\Ra \right)\\
&+\tau^2\left(\sum\theta_k\La \sd U_k^2 x_0,x_0\Ra +\sum\theta_k\La \sd x_k,\sd x_k\Ra -2\sum\theta_k\La \sd U_k x_0,\sd x_k\Ra\ci\cH \right)+ o(\tau^2)
\end{align*}
Notice that $\sum \theta_k\sd x_k=\sum \theta_k(x_k-x_0)=0$ and also $\sum \theta_k\sd U_k=-\sd U$. Hence
\begin{align}
\label{phi-02}
\varPhi(\tau)=\tau\La \sd U x_0,x_0\Ra +\tau^2\sum\theta_k\left(\left\|\sd U_kx_0\right\|^2+\left\|\sd x_k\right\|^2-2\La \sd U_k x_0,\sd x_k\Ra \right)+ o(\tau^3) ,
\end{align}
so $\varPhi'(0)=\La \sd U x_0,x_0\Ra$, $\varPhi''(0)\ge 0$ (the function $\varPhi$ is clearly real analytic as a composition of two real analytic functions, so all derivatives are well defined). 

Note, that the above formula \eqref{phi-02} holds for any choice of parameters  $U_k$, $x_k$ satisfying \eqref{mart-prop-01} with $U_0=\I$. Let us apply it to the case of $3$ points in the domain of $\phi$ with $x_0=(x_1+x_2)/2$, $\I=U_0=(U_1+U_2)/2$ (so $\sd U=0$). Recalling the definition of $\varPhi$, we see that in this case the  condition $\varPhi''(0)\ge 0$ means that the second differential of $\phi$ is non-negative in the corresponding direction at $U=\I$ and arbitrary $x$ (the function $\phi$ is real analytic, so all the differentials are well defined).

But we can pick any direction, so the second differential of $\phi$ is non-negative at $U=\I$ (and arbitrary $x$). 
The change of variables \eqref{ChVar-01} then implies that the second differential of $\phi$ is non-negative everywhere. This implies, in particular, that \emph{for any choice} of parameters $U_k$, $x_k$ (satisfying \eqref{mart-prop-01} with $U_0=\I$) we get  $\varPhi''(\tau)\ge0$, so $\varPhi $ is convex.

%Returning to the general choice of arguments $U$, $x$, we can see from \eqref{phi-02} that 
%\[\varPhi'(0)= \La \sd U x_0, x_0\Ra\ci\cH.\] 

So, since for such choice of parameters $U_k$, $x_k$ the function $\varPhi $ is convex and $\varPhi(0)=0$, 
we conclude that 
\begin{align*}
\varPhi(1) \ge \varPhi'(0) = \La \sd U x_0, x_0\Ra\ci\cH, 
\end{align*}
and $\varPhi(1)$ is exactly the left hand side of \eqref{key-02}. 
%%  
%
%When $\sd U= 0$,
%\[
%\varPhi(\tau)=\tau^2\sum\theta_k\left(\left\|\sd U_kx_0\right\|^2+\left\|\sd x_k\right\|^2\right)+ o(\tau^4)
%\]
%
%so, in particular,
%\[
%\varPhi(1)= \sum_k \theta_k \phi(U_k, x_k) - \phi(\I,x_0) \ge 0,
%\]
%which implies that $\phi$ is convex at $U_0=I$.  By the change of variables discussed above, it is convex at any point.
%
%When $\sd U\neq 0$,
%\[
%\varPhi'(\tau)=\La \sd U x_0,x_0\Ra ,
%\]
%and hence
%\[
%\varPhi(1)-\varPhi(0)= \sum_k \theta_k \phi(U_k, x_k) - \phi(\I,x_0)\geq \La \sd U x_0,x_0\Ra .
%\]
%
%As before, by the same change of variables, the lemma is proved for all $U$.
\end{proof}

\end{document}